\documentclass[a4paper,11pt,reqno]{amsart}

\usepackage{graphicx}
\usepackage{mathtools}
\usepackage{hyperref}
\usepackage{float} 
\usepackage{amsfonts}
\usepackage{amsthm}
\usepackage{latexsym}
\usepackage{amsmath}
\usepackage{amssymb}
\usepackage{amscd}
\usepackage{epsf}
\usepackage{tikz}
\usetikzlibrary{matrix,arrows}

\usepackage[english]{babel}
\usepackage{lmodern}

\usepackage[utf8]{inputenc}

\usepackage[T1]{fontenc}
\newtheorem{thm}{Theorem}

\newtheorem{lemma}{Lemma}

\newtheorem{corollary}{Corollary}
\newtheorem{example}{Example}
\newtheorem*{remark}{Remark}

\newcommand{\nc}{\newcommand}
\nc{\Ra}{\Rightarrow}
\nc{\ra}{\rightarrow}
\nc{\ora}{\overrightarrow}
\nc{\ol}{\overline}
\nc{\de}{\delta}
\nc{\De}{\Delta}
\nc{\p}{\partial}
\nc{\ex}{\exists}
\nc{\nex}{\nexists}
\nc{\fa}{\forall}
\nc{\Ea}{\Leftrightarrow}
\nc{\ea}{\leftrightarrow}
\nc{\ub}{\underbrace}
\nc{\Rn}{\mathbb{R}^n}
\nc{\Rne}{\mathbb{R}^{n+1}}
\nc{\Rnm}{\mathbb{R}^{n-1}}
\nc{\bR}{\mathbb{R}}
\nc{\bC}{\mathbb{C}}
\nc{\bN}{\mathbb{N}}
\nc{\bZ}{\mathbb{Z}}
\nc{\bfx}{\textbf{\textit{x}}}
\nc{\bfX}{\textit{\textbf{X}}}
\nc{\bfY}{\textbf{\textit{Y}}}
\nc{\bfZ}{\textbf{\textit{Z}}}
\nc{\bfy}{\textbf{\textit{y}}}
\nc{\bfa}{\textbf{\textit{a}}}
\nc{\bfv}{\textbf{\textit{v}}}
\nc{\D}{\Delta}
\nc{\Dn}{\Delta^n}
\nc{\lx}{\lambda(x)}
\nc{\ba}{\textbf{\textit{a}}}
\nc{\Sn}{\S^n}
\nc{\Snm}{\S^{n-1}}
\nc{\Sne}{\S^{n+1}}
\nc{\na}{\nabla}
\nc{\deri}[2]{\frac{d #1}{d #2}}
\nc{\ph}[1]{\phi\left( #1\right)}
\nc{\ft}{_{n=1}^\infty}
\nc{\gs}{\sum_{n=1}^\infty}
\nc{\qf}[2]{\langle #1 #2,#2\rangle}
\nc{\mc}[1]{\mathcal{#1}}
\nc{\conj}[1]{\overline{#1}}
\everymath{\displaystyle}

\title{Rational inner functions and their Dirichlet type norms}
\author{Linus Bergqvist}
\date{}

\begin{document}

\maketitle

\begin{abstract}
We study membership of rational inner functions in Dirichlet-type spaces in polydisks. In particular, we prove a theorem relating such inclusions to $H^p$ integrability of partial derivatives of a RIF, and as a corollary we prove that all rational inner functions on $\mathbb{D}^n$ belong to $\mathcal{D}_{1/n, \ldots ,1/n}(\mathbb{D}^n)$. Furthermore, we show that if $1/p \in \mathcal{D}_{\alpha,...,\alpha}$, then the RIF $\tilde{p}/p \in \mathcal{D}_{\alpha+2/n,...,\alpha+2/n}$. Finally we illustrate how these results can be applied through several examples, and how the Łojasiewicz inequality can sometimes be applied to determine inclusion of $1/p$ in certain Dirichlet-type spaces.
\end{abstract}

\section{Introduction}

In this paper, we address aspects of the classical problem of determining how boundary singularities of holomorphic functions defined on domains of $\mathbb{C}^n$ affect their integrability and the integrability of their derivatives. In particular, we will do this for a special class of rational functions defined on the $n$-dimensional polydisc
$$
\mathbb{D}^n = \{z: |z_j|<1, j=1 \ldots n \},
$$
namely rational inner functions, or RIFs for short. A bounded holomorphic function $f: \mathbb{D}^n \ra \mathbb{C}$ is said to be inner if $|f(z)| = 1$ for almost every $z \in \mathbb{T}^n$, and a rational inner function is, as the name implies, a rational function $f=q/p$ where $q,p \in \mathbb{C}[z_1 , \ldots, z_n]$, which is also inner. Inner functions are very important in function theory on the unit disc, among other reasons, because of the classical inner-outer factorization of one variable functions in the Hardy space. Although such a factorization does not exist in higher dimensions, RIFs still play an important role, for example since bounded holomorphic functions on $\mathbb{D}^n$ can be approximated locally-uniformly by constant multiples of RIFs (see \cite[Theorem $5.2.5$ $b)$ and $5.5.1$]{Rudin}).

In this paper we are interested in determining integrability properties of higher order partial derivatives of RIF:s, and since a rational function whose denominator has no zeros on $\overline{\mathbb{D}^n}$ will be smooth on $\mathbb{D}^n$, the interesting cases are the ones where the denominator $p(z)$ vanishes on the boundary of $\mathbb{D}^n$. Since we are dealing with inner functions, it turns out that the numerator and denominator must vanish at the same points, that is $Z(p) \cap \overline{\mathbb{D}^n} = Z(q) \cap \overline{\mathbb{D}^n}$. However since we are working with several variables, this often happens despite $p$ and $q$ not sharing a common factor, which severely complicates the question of determining boundary regularity and integrability of both the RIF and its derivatives. When we are talking about the singular set of a RIF we will mean the zero set of its denominator on the boundary of $\mathbb{D}^n$. 

In recent years, several authors have studied integrability and regularity of rational functions on the polydisc. There are many different ways to do this, and to define regularity; one could, for example, investigate area integrability or boundary integrability of the function, or the same for the partial derivatives up to some given order. In \cite{KNE}, Knese gives an algebraic classification of the ideal of polynomials $p$ such that $q/p$ lies in $L^2(\mathbb{T}^2)$ for a fixed polynomial $q$, and furthermore proves that rational inner functions have non-tangential limits everywhere on the $n-$torus. In \cite{LMS} and \cite{level}, a thorough analysis of the unimodular level sets and singular sets of RIFs gives a fairly complete understanding of the $L^p$ integrability on $\mathbb{T}^2$ of the partial derivatives of RIFs in $2$ variables. For instance it is shown that $\partial_{z_1} \phi$ and $\partial_{z_2} \phi$ have the same integrability properties. Furthermore, this is used to find a lower bound of $\alpha$ for inclusion for the RIF in certain Dirichlet type spaces $\mathcal{D}_{\alpha, \alpha}$.

However, it is shown in \cite{SRIF} that a lot of important properties are lost in higher dimensions. For example, in two variables, the singular sets are isolated points, and the unimodular level sets passing through these singularities are curves parametrizable by analytic functions (See \cite{level}). In dimension $d \geq 3$  this is no longer true; the singular set is only guaranteed to be contained in an algebraic set of dimension $n-2$; so for example in dimension $3$ the singular set can be a union of points and curves (see Prop. 3.6 in \cite{AgMcS}), and the level sets are no longer guaranteed to even be continuous. This makes it difficult to find higher dimensional equivalents of many of the methods used in \cite{LMS}. Furthermore, in \cite{SRIF} it is also shown that in higher dimensions it is no longer true that all partial derivatives must have the same $H^p$ integrability.

Despite these limitations, in this article we prove certain results regarding regularity of the mixed partial derivative of RIFs in arbitrary dimension by proving inclusion for RIFs in certain Dirichlet type spaces; a problem more difficult than that addressed in \cite{LMS} in the sense that we are dealing with both higher dimensions and higher order derivatives. Recall that $\mathcal{D}_{\alpha_1, \ldots, \alpha_n}(\mathbb{D}^n)$ is the space of holomorphic functions $f(z_1, \ldots, z_n)$  on $\mathbb{D}^n$ for which
\begin{equation} \label{pow_norm}
\| f \|_{\alpha_1, \ldots, \alpha_n}^2 := \sum_{k_n=0}^\infty \ldots \sum_{k_1=0}^\infty (1+k_1)^{\alpha_1} \cdots (1+k_n)^{\alpha_n} |\hat{f}(k_1, \ldots k_n)|^2
\end{equation}
is finite, which is equivalent to requiring that
$$
\int_{\mathbb{D}^n} \left| \frac{\partial^n}{\partial z_1 \cdots \partial z_n} (z_1 \cdots z_n f(z))\right|^2 dA_{\alpha_1}(z_1) \cdots dA_{\alpha_n}(z_n) < \infty,
$$
where $dA_{\alpha_j}(z_j) = (1-|z_j|)^{1-\alpha_j} dA(z_j)$.

In particular, we show that for an RIF $\phi$, we have that
$$
\frac{\partial \phi}{\partial z_k} \in H^{\alpha_k}(\mathbb{D}^n), \text{ for } k=1, \ldots, n  \Ra \phi \in \mathcal{D}_{c_1 \alpha_1, \ldots, c_n \alpha_n},
$$
where $\sum_{i=1}^n c_i = 1$.
In particular, since all partial derivatives of all RIFs lie in $H^1(\mathbb{D}^n)$ (see \cite{SRIF}), this implies that all RIFs lie in $\mathcal{D}_{1/n, \ldots, 1/n}$, and so, they all have strictly better integrability properties than just $H^2$-regularity. Also, by applying the above result to RIFs whose partial derivatives have different $H^p$ integrability, we demonstrate through examples how this result can be applied to obtain better regularity in certain directions at the expense of regularity in the others. 
Finally we show that for a polynomial $p$ with no zeros in the polydisc, if $1/p \in \mathcal{D}_{\overline{\alpha}}$, then the RIF $ \phi = \tilde{p}/p$ corresponding to $p$ must lie in $\mathcal{D}_{\overline{\alpha}+2/n}$. This essentially means that the improvement in Dirichlet type integrability for a RIF obtained by the zeros of $\tilde{p}$ mitigating the zeros of $p$ will always correspond to at least a parameter increase of $2/n$ in each variable. Also, since the Dirichlet type norm of $1/p$ is often easier to calculate, this gives a useful method of determining a lower bound on the regularity of RIFs.

\section{Preliminaries and notation}

\subsection{Dirichlet type spaces}

\bigskip

Recall that the \emph{Dirichlet type spaces} $\mathcal{D}_{\overline{\alpha}}(\mathbb{D}^n)$ are reproducing kernel Hilbert spaces with reproducing kernel given by
$$
\sum_{k_1, \ldots, k_n=0}^\infty \frac{z_1^{k_1} \cdots z_n^{k_n}}{(k_1+1)^{\alpha_1} \cdots {k_n+1}^{\alpha_n}},
$$
consisting of holomorphic functions $f: \mathbb{D}^n \ra \mathbb{C}$ whose norm, $\| f \|_{\overline{\alpha}}$, given by \eqref{pow_norm} where $\overline{\alpha} = (\alpha_1, \ldots, \alpha_n) \in \mathbb{R}^n$, is finite. For a scalar $\alpha \in \mathbb{R}$, we will denote by $\mathcal{D}_\alpha(\mathbb{D}^n)$ the \emph{isotropic Dirichlet type space} $\mathcal{D}_{\alpha, \ldots, \alpha}(\mathbb{D}^n)$. Furthermore, for $\overline{\alpha}$ with all $\alpha_i \leq 0$ an equivalent norm for $\mathcal{D}_{\overline{\alpha}}$ is given by
$$
\int_{\mathbb{D}^n} |f(z)|^2 \prod_{i=1}^n (1 - |z_i|^2)^{-1-\alpha_i} dA(z)
$$
(see \cite{Kap}). Note that for $\alpha = -1$, $\mathcal{D}_\alpha$ is the Bergman space, for $\alpha=0$, we get $H^2$, and for $\alpha=1$ we get the ordinary Dirichlet space. Furthermore, if all $\alpha_j > 1$, then $\mathcal{D}_{\overline{\alpha}}$ is in fact an algebra, and if all $\alpha_j \leq 0$, then the multiplier algebra of $\mathcal{D}_{\overline{\alpha}}$ is $H^\infty(\mathbb{D}^n)$ (see \cite{JupRed}). This will be used in the proof of Theorem \ref{LMS11.3}. For $\overline{\alpha}$ with $0 < \alpha_j \leq 1$, the situation is more complicated, and $H^\infty(\mathbb{D}^n)$ is a proper subset of the multiplier algebra (See \cite{Perfekt} for more about the multiplier algebra of Dirichlet type spaces). 

From the power series norm, we immediately have that polynomials are dense in all Dirichlet type spaces, and we see that if $\alpha_i \leq \beta_i$ for $i  = 1, \ldots n$, then $\mathcal{D}_{\overline{\beta}} \subset \mathcal{D}_{\overline{\alpha}}$. 

Furthermore, we will denote by $D_\alpha(\mathbb{D})$ the one variable Dirichlet type space with parameter $\alpha$, and the norm by $\| f \|_{D_\alpha}$, and we will denote by $\| f \|_{p;k}$ the norm $\| f \|_{\alpha_1, \ldots. \alpha_n}$, where $\alpha_k = p,$ and $\alpha_i = 0$ for all $i \neq k$, and $\mathcal{D}_{(p;k)}(\mathbb{D}^n)$ is the Dirichlet type space corresponding to this norm. Also, for $z = (z_1, \ldots, z_n)$, and $\hat{z} \in \mathbb{T}^{n-1}$, we will denote by $f_{\hat{z}}(z_k)$ the one variable function obtained by keeping $(z_1, \ldots, z_{k-1}, z_{k+1}, \ldots z_n) = \hat{z}$ fixed. Finally, we will denote by $(z_k; \hat{z})$ the point $(z_1, \ldots, z_k, \ldots z_n)$, where
$$
(z_1, \ldots, z_{k-1}, z_{k+1}, \ldots z_n) = \hat{z}.
$$

\subsection{Rational Inner Functions}

\bigskip

It is known that a RIF defined on $\mathbb{D}^n$ is necessarily of the form
$$
c z_1^{k_1} \cdots z_n^{k_n} \frac{\tilde{p}}{p}
$$ where $p(z)$ is a polynomial of multi-degree $(d_1, \ldots, d_n)$ with no zeros in the polydisc, $\tilde{p}(z)$ is the reflection of $p$ given by 
$$
\tilde{p}(z) := z_1^{d_1} \cdots z_n^{d_n} \overline{p \left(\frac{1}{z_1}, \ldots, \frac{1}{z_n} \right)},
$$  
and $c$ is a unimodular constant. (See \cite[Thm 5.2.5 $a)$]{Rudin}). 

For the remainder of this paper, we will only work with RIFs of the form $\tilde{p}/p$, but the same results for RIFs with monomial factors follow immediately once the results for RIFs without monomial factors are known. Also, we will assume that the polynomial $p$ is \emph{atoral}, which for the purpose of this article means that $p$ is not constant in any variable (all variables are represented in the polynomial) and that $\tilde{p}$ and $p$ have no common factors. Also, as stated in the introduction, we will assume that $p(z)$ has zeros on $\overline{\mathbb{D}^n}$ since otherwise the results are trivial. In fact, for such a polynomial $Z(p) \cap \overline{\mathbb{D}^n} = Z(p) \cap \mathbb{T}^n$ since a zero on $\overline{\mathbb{D}^n} \setminus \mathbb{T}^n$ would imply that $p$ and $\tilde{p}$ has a common factor (see Lemma $10.1$ in \cite{KNE}). This is one of the reasons why determining inclusion in $L^p(\mathbb{T}^n)$ is more interesting in this context than determining inclusion in $L^p(\partial \mathbb{\mathbb{D}}^n)$.

Note that a RIF in one variable is just a finite Blaschke product, and since for a RIF $\phi = \overline{p}/p$ on $\mathbb{D}^n$ where $p$ is a polynomial of multi-degree $(d_1, \ldots, d_n)$, we have that the one variable function $\phi_{\hat{z}}(z_k)$ is both inner and rational, it will generically be a Blaschke product of degree as $d_k$.

\section{General results for RIF:s}

In this section we prove a Theorem relating $H^p$ integrability of the partial derivatives of a RIF with inclusion in certain Dirichlet type spaces of the RIF itself. As a corollary, we see that all RIF:s on $\mathbb{D}^n$ lie in $\mathcal{D}_{1/n}(\mathbb{D}^n)$. 

\begin{lemma} \label{embe}
Let $f \in \mathcal{D}_{\alpha_1, \ldots, \alpha_n}(\mathbb{D}^n)$ and let $p_j > 1$ for $j = 1, \ldots, n$. Then 
$$
\| f \|_{\alpha_1, \ldots, \alpha_n}^2 \leq \prod_{i=1}^n \| f \|_{c_i \alpha_i; i}^{1/c_i},
$$
where $c_n = \prod_{j=1}^{n-1} p_j/(p_j - 1)$ and $c_l = p_l \prod_{j=1}^{l-1} p_j/(p_j - 1)$ for $l \neq n$.
\end{lemma}
\begin{remark}
\normalfont
Note that the map ($p_1, \ldots p_{n-1}) \mapsto (c_1, \ldots c_n)$ is a bijection from $\{(p_1, \ldots, p_{n-1}): p_i > 1 \quad i=1, \ldots, n-1 \}$ to
$$
\left\{ (c_1, \ldots, c_n): 1 = \sum_{i=1}^n c_i^{-1} \text{ and } c_i > 0 \quad i=1, \ldots, n \right\},
$$
and so the above inequality is true for every choice of $(c_1, \ldots, c_n)$ in the above set.
\end{remark}
\begin{proof}
By applying Hölder's inequality with $p = p_1$ and $q = p_1/(p_1-1)$, we get that
\begin{multline*}
\| f \|_{\alpha_1, \ldots, \alpha_n}^2 = \sum_{k_1 = 0}^\infty \ldots \sum_{k_n = 0}^\infty (1+k_1)^{\alpha_1} \cdots (1+k_n)^{\alpha_n} |a_{k_1, \ldots , k_n}|^2 \\
= \sum_{k_1 = 0}^\infty \ldots \sum_{k_n = 0}^\infty (1+k_1)^{\alpha_1} |a_{k_1, \ldots , k_n}|^{2 (1/p_1)} (1+k_2)^{\alpha_2} \cdots (1+k_n)^{\alpha_n} |a_{k_1, \ldots , k_n}|^{2(1 - 1/p_1)} \\
\leq \|f\|_{p_1; 1}^{1/p_1} \left(\sum_{k_2 = 0}^\infty \ldots \sum_{k_n = 0}^\infty (1+k_2)^{q_1 \alpha_2} \cdots (1+k_n)^{q_1 \alpha_n} |a_{k_1, \ldots , k_n}|^2 \right)^{1/q_1}.
\end{multline*}
By applying the same procedure with Hölder's inequality to the sum in parenthesis, but with $p_2$ instead, we see that
\begin{multline*}
\left(\sum_{k_2 = 0}^\infty \ldots \sum_{k_n = 0}^\infty (1+k_2)^{q_1 \alpha_2} \cdots (1+k_n)^{q_1 \alpha_n} |a_{k_1, \ldots , k_n}|^2 \right)^{1/q_1} \\
\leq \| f \|_{p_2 q_1; 2}^{\frac{1}{p_2 q_1}} \left(\sum_{k_3 = 0}^\infty \ldots \sum_{k_n = 0}^\infty (1+k_3)^{q_2 q_1 \alpha_3} \cdots (1+k_n)^{q_2 q_1 \alpha_n} |a_{k_1, \ldots , k_n}|^2 \right)^{\frac{1}{q_1 q_2}}.
\end{multline*}
By doing this inductively for $p_3, \ldots p_{n-1}$, we see that
\begin{multline*}
\| f \|_{\alpha_1, \ldots, \alpha_n}^2 \\
\leq \| f \|_{p_1 \alpha_1; 1}^{1/p_1} \| f \|_{p_2 q_1 \alpha_2; 2}^{1/(p_2 q_1)} \| f \|_{p_3 q_2 q_1 \alpha_3; 3}^{1/(p_3 q_2 q_1)} \cdots \| f \|_{p_{n-1} q_{n-2} \cdots q_1 \alpha_{n-1}; n-1}^{1/(p_{n-1} q_{n-2} \cdots q_1)} \| f \|_{q_{n-1} \cdots q_1 \alpha_n; n}^{1/(q_{n-1} \cdots q_1)}.
\end{multline*}
By using $q_j = p_j/(p_j-1)$, this finishes the proof.
\end{proof}

\begin{corollary}
Let $f \in \mathcal{D}_{\alpha_1, \ldots, \alpha_n}(\mathbb{D}^n)$. Then
$$
\| f \|_{\alpha_1, \ldots, \alpha_n}^2 \leq \prod_{i=1}^n \| f \|_{n \alpha_i; i}^{1/n}.
$$
\end{corollary}
\begin{proof}
Apply Lemma \ref{embe} with $p_i = n + 1 - i$ for $i = 1, \ldots, n-1$. Then each product defining $c_j$ will telescope, and we get $c_j=n$ for every $j = 1, \ldots n$, which gives the desired inequality.
\end{proof}

In order to apply this to show inclusion of certain RIF:s in specific Dirichlet spaces, we need a few additional lemmas.

\begin{lemma} \label{onedim}

Let $b(z):= \prod_{j=1}^n b_{\alpha_j}(z)$ be a finite Blaschke product, where
$$
b_{\alpha_j}(z) = \frac{z- \alpha_j}{1- \overline{\alpha}_j z}, \quad \text{for }\alpha_j \in \mathbb{D}.
$$
Then $b$ satisfies
$$
\|b\|_{D_\mathfrak{p}}^2 \lesssim \epsilon(b)^{1-\mathfrak{p}} \quad \text{for } 0 < \mathfrak{p} < \infty ,
$$
where $\epsilon(b) := \min \{1- |\alpha_j|: 1 \leq j \leq n \}$ is the distance from the zero set of $b$ to $\mathbb{T}$ and the implied constant depends on $\mathfrak{p}$ and $\deg b$ = $n$.

\end{lemma}

The case where $1 \leq \mathfrak{p} \leq 2$ is Lemma $9.3$ in \cite{LMS}, but as the authors of that paper point out, the more general case formulated above also holds. We will now show how Lemma \ref{onedim} can be obtained from Lemma $9.3$ in \cite{LMS}. 

\begin{proof}
First of all, the proof given in \cite{LMS} works for $0< \mathfrak{p} \leq 2$, since the only time $1 \leq \mathfrak{p}$ is used is in the inequality
$$
\epsilon(b)^{1 - \mathfrak{p}} := (\min \{ 1 - |\alpha_j| \})^{1 - \mathfrak{p}} = \max \left\{ 1 - |\alpha_j|^{1 - \mathfrak{p}} \right\} \approx \sum_{j=1}^n (1-|\alpha_j|)^{1- \mathfrak{p}},
$$
on page $314$. However, for $\mathfrak{p} > 0$ we still have that
$$
\epsilon(b)^{1 - \mathfrak{p}} := (\min \{ 1 - |\alpha_j| \})^{1 - \mathfrak{p}} \lesssim \sum_{j=1}^n (1-|\alpha_j|)^{1- \mathfrak{p}},
$$
and in fact, for the remainder of that proof only inequalities are used, and so the same proof works verbatim with
$$
\epsilon(b)^{1 - \mathfrak{p}} \lesssim \sum_{j=1}^n (1-|\alpha_j|)^{1- \mathfrak{p}}
$$
instead of
$$
\epsilon(b)^{1 - \mathfrak{p}} \approx \sum_{j=1}^n (1-|\alpha_j|)^{1- \mathfrak{p}}.
$$
We will now prove the statement for general $\mathfrak{p}$ by using higher derivatives and reducing the statement to the case $0 < \mathfrak{p} \leq 2$. 

First, consider a single Blaschke factor $b_{\alpha}(z) = \frac{z- \alpha}{1- \overline{\alpha} z}$. Its derivative is given by
$$
b'_{\alpha}(z) = \frac{|\alpha|^2 - 1}{(1- \overline{\alpha} z)^2},
$$
and so
$$
b^{(n+1)}_{\alpha}(z) = \frac{n! (|\alpha|^2 - 1)(- \overline{\alpha})^n}{(1- \overline{\alpha} z)^{2+n}}.
$$
This means that
$$
\left|b^{(n+1)}_{\alpha}(z) \right| \leq \frac{C}{\epsilon(b_\alpha)^{n+1}} |b_\alpha(z)|,
$$
and thus
$$
\| b^{(n)}_\alpha \|_{D_{\mathfrak{p}}}^2 \leq \frac{C}{\epsilon(b_\alpha)^{2n}} \| b_\alpha\|_{D_{\mathfrak{p}}}^2.
$$
In the above equations the constant $C$ depends only on $\alpha$ and $n$, and is uniformly bounded in $\alpha$, but not in  $n$.

By applying $\| f \|_{D_{\mathfrak{p}+2}} \approx \| f' \|_{D_{\mathfrak{p}}}$ inductively, we see that for $\mathfrak{p} = \mathfrak{p}'+2k$, where $k$ is an integer and $0 < \mathfrak{p}' \leq 2$, we have that
$$
\| b_\alpha \|^2_{D_{\mathfrak{p}}} \cong \| b_\alpha^{(k)} \|_{D_{\mathfrak{p}'}}^2 \leq \frac{C}{\epsilon(b_\alpha)^{2k}} \| b_\alpha\|_{D_{\mathfrak{p}'}}^2 \leq \frac{C}{\epsilon(b_\alpha)^{2k}} \epsilon(b_\alpha)^{1-\mathfrak{p}'} \cong \epsilon(b_\alpha)^{1- \mathfrak{p}},
$$ 
which proves the statement for a single Blaschke factor.

Finally, for a general Blaschke product, the $k$:th derivative will be a linear combination of expressions of the form
$$
\prod_{j=0}^m b^{(n_j)}_{\alpha_j}(z),
$$
where $\sum_j n_j = k$. Since each factor satisfies a bound of the form 
$$
\left|b^{(n_j)}_{\alpha_j}(z)\right| \leq \frac{C}{\epsilon(b_{\alpha_j})^{n_j}} |b_{\alpha_j}(z)|
$$
the absolute value of each term will be bounded by $\frac{C_1}{\epsilon(b)^{k}} |b(z)|$, and thus the entire linear combination will be bounded by
$$
\frac{C_2}{\epsilon(b)^{k}} |b(z)|.
$$
We can thus apply the same induction argument as for a single Blaschke factor and the full statement follows.

\end{proof}

We will now prove a higher dimensional equivalent of the result for RIF:s in two variables which is proved in Proposition 4.4 in \cite{LMS} (see also \cite{SRIF}).

\begin{lemma} \label{h1}

Let $\phi = \tilde{p}/p$ be an RIF on $\mathbb{D}^n$ with $\deg \phi = (m_1, \ldots m_n)$. Then
$$
\left\Vert \frac{\partial \phi}{\partial z_i} \right\Vert_{H^1(\mathbb{D}^n)} = m_i. 
$$
\end{lemma}

\begin{remark}
\normalfont
Note in particular that this implies that all RIF:s have partial derivatives that lie in $H^1(\mathbb{D}^n)$.
\end{remark}

\begin{proof}
 Fix $\hat{z} \in \mathbb{T}^{n-1}$ and consider the finite Blaschke product $\phi_{\hat{z}}(z_k)$. Furthermore, $\deg \phi_{\hat{z}} = m_k$ except on a set of Lebesgue measure zero on $\mathbb{T}^{n-1}$, corresponding to the points $\hat{z} \in \mathbb{T}^{n-1}$ for which $\phi$ has a singularity for some value of $z_k$.
 Since this is a finite Blaschke product, we have that
 $$
 \left| \frac{\partial \phi}{\partial z_k}(z_1, \ldots, z_k) \right| = |\phi_{\hat{z}}'(z_k)| = \frac{\phi_{\hat{z}}'(z_k)}{\phi_{\hat{z}}(z_k)}z_k, \quad \text{for } z_1 \in \mathbb{T}. 
 $$
 So by the argument principle
 \begin{multline*}
 \left\Vert \frac{\partial \phi}{\partial z_k} \right\Vert_{H^1(\mathbb{D}^n)} = \frac{1}{(2 \pi)^n}\int_{\mathbb{T}^{n-1}} \left( \int_\mathbb{T} \frac{\phi_{\hat{z}}'(z_k)}{\phi_{\hat{z}}(z_k)}z_k |d z_k|\right) |d\hat{z}| \\
 = \frac{1}{(2 \pi)^{n-1}}\int_{\mathbb{T}^{n-1}} \left( \frac{1}{2 \pi i}\int_\mathbb{T} \frac{\phi_{\hat{z}}'(z_k)}{\phi_{\hat{z}}(z_k)} d z_k\right) |d\hat{z}| = \frac{1}{(2 \pi)^{n-1}}\int_{\mathbb{T}^{n-1}} m_k |d\hat{z}| = m_k.
 \end{multline*}
 
\end{proof}

Now, for a general RIF $\phi = \tilde{p}/p$ with $\deg \phi = (d_1, \ldots, d_n)$, the one variable function $\phi_{\hat{z}}(z_k)$ is a finite Blaschke product with $n_{\hat{z};k} := \deg \phi_{\hat{z}} \leq d_k$. Now define $\delta(\phi, \hat{z})$ as the distance from $\mathcal{Z}_{\tilde{p}} \cap \{(z_k; \hat{z}): z_k \in \mathbb{D} \}$ to $\mathbb{T}^n$. That is, if $\alpha_1, \ldots \alpha_{n_{\hat{z};k}} \in \mathbb{D}$ are the zeros of $\phi_{\hat{z}}(z_k)$, then
$$
\delta(\phi, \hat{z}) = \min_{1 \leq i \leq n_{\hat{z};k}} |1 - \alpha_i|.
$$
Now for each $x > 0$, we define
$$
\Omega_x = \{ \hat{z} \in \mathbb{T}^{n-1}: \delta(\phi, \hat{z}) < 1/x \}.
$$
The following lemma relates containment of $\partial_{z_k} \phi$ in $H^p$ to the the size of these sets (parametrized by $x$).

\begin{lemma}{[Theorem 2.1 of \cite{SRIF}]} \label{levelset}

Let $\phi = \tilde{p}/p$ be a RIF on $\mathbb{D}^n$. Then for $1 \leq \mathfrak{p} < \infty$, $\frac{\partial \phi}{\partial z_k} \in H^{\mathfrak{p}}(\mathbb{D}^n)$ if and only if
$$
\int_1^\infty \mu(\Omega_x) x^{\mathfrak{p}-2} dx < \infty.
$$
\end{lemma}

And this in turn can be used to show that for an RIF, having partial derivatives in $H^p$ implies containment in suitable Dirichlet type spaces.

\begin{lemma} \label{hp}
Let $\phi = \tilde{p}/p$ be an RIF on $\mathbb{D}^n$. Then for $0 < \mathfrak{p} < \infty$, if $\frac{\partial \phi}{\partial z_k} \in H^{\mathfrak{p}}(\mathbb{D}^n)$, then $\phi \in \mathcal{D}_{\mathfrak{p};k}(\mathbb{D}^n)$.
\end{lemma}
\begin{proof}
By Lemma \ref{h1} it suffices to prove the statement for $1 \leq \mathfrak{p} < \infty$.

We denote by $(l;k)$ the multi-index $(l_1, \ldots, l_{k-1}, l_{k+1}, \ldots l_n)$. Now for
$$
\phi(z) = \sum_{l \in \mathbb{N}^n} a_l z_1^{l_1} \cdots z_n^{l_n},
$$
define a sequence of functions, by
$$
\phi_{l_k}(\hat{z}) = \sum_{(l;k) \in \mathbb{N}^{n-1}} a_l \hat{z}^{(l;k)}, \quad \hat{z} \in \mathbb{D}^{n-1}.
$$
Here $l \in \mathbb{N}^n$ is understood to be obtained by inserting $l_k$ in position $k$ of $(l;k) \in \mathbb{N}^{n-1}$.
That is
$$
\phi(z) = \sum_{l_k = 0}^\infty \phi_{l_k}(\hat{z}) z_k^{l_k}.
$$
This means that
$$
\Vert \phi_{\hat{z}}(z_k) \Vert_{D_\alpha}^2 = \sum_{l_k=0}^\infty (1+l_k)^\alpha |\phi_{l_k}(\hat{z})|^2 = \sum_{l_k=0}^\infty (1+l_k)^\alpha \left| \sum_{(l;k) \in \mathbb{N}^{n-1}} \hat{z}^{(l;k)} |a_l| \right|^2.
$$

By putting this together we get that
\begin{multline*}
\Vert \phi \Vert_{\mathfrak{p};k}^2 = \sum_{l_k = 0}^\infty (1+l_k)^\mathfrak{p} \left( \sum_{(l;k) \in \mathbb{N}^{n-1}} |a_{(l;k)}|^2 \right) = \sum_{l_k = 0}^\infty (1+l_k)^\mathfrak{p} \Vert \phi_{l_k} \Vert_{H^2(\mathbb{D}^n)}^2
\\
= \frac{1}{(2 \pi)^{n-1}} \int_{\mathbb{T}^{n-1}} \sum_{l_k=0}^\infty (1+l_k)^\mathfrak{p} |\phi_{l_k}(\hat{z})|^2 |d \hat{z}| = \frac{1}{(2 \pi)^{n-1}} \int_{\mathbb{T}^{n-1}} \Vert \phi_{\hat{z}}(z_k) \Vert_{D_\mathfrak{p}}^2 |d \hat{z}|.
\end{multline*}
By applying Lemma \ref{onedim} to the above calculations and integrating over level sets, we get that
\begin{multline*}
\Vert \phi \Vert_{\mathfrak{p};k}^2 \lesssim \frac{1}{(2 \pi)^{n-1}} \int_{\mathbb{T}^{n-1}} \delta(\phi, \hat{z})^{(1-\mathfrak{p})} |d \hat{z}| \\
= \left(\int_0^\infty t^{-\mathfrak{p}} \mu (\{\hat{z} \in \mathbb{T}^{n-1}: \delta(\phi, \hat{z}) < t \}) dt \right)^{1-\mathfrak{p}} \\
= \left(\int_0^\infty x^{\mathfrak{p}-2} \mu (\{\hat{z} \in \mathbb{T}^{n-1}: \delta(\phi, \hat{z}) < 1/x \}) dx \right)^{1-\mathfrak{p}} \\
= \left( \int_0^\infty x^{\mathfrak{p}-2} \mu(\Omega_x) dx \right)^{1-\mathfrak{p}}.
\end{multline*}
Since $\mu(\Omega_x)$ is always bounded, the last integral is finite if and only if 
$$
\int_1^\infty x^{\mathfrak{p}-2} \mu(\Omega_x) dx  < \infty,
$$
and by Lemma \ref{levelset} this is true if (and only if) $\frac{\partial \phi}{\partial z_k} \in H^p(\mathbb{D}^n)$.
Put together, this shows that $\frac{\partial \phi}{\partial z_k} \in H^\mathfrak{p}(\mathbb{D}^n)$ implies that
$$
\phi \in \mathcal{D}_{(\mathfrak{p};k)}(\mathbb{D}^n),
$$
which finishes the proof.
\end{proof}

\begin{thm} \label{hp_embe}
Suppose $\frac{\partial \phi}{\partial z_k} \in H^{\alpha_k}(\mathbb{D}^n)$, where $0 < \alpha_k \leq 2$ for $k = 1, \ldots n$. Then for every choice of $(c_1, \ldots, c_n)$ with $c_j > 0$ satisfying $1 = \sum_{i=1}^n c_i^{-1}$, we have that $\phi \in \mathcal{D}_{\frac{\alpha_1}{c_1}, \ldots, \frac{\alpha_n}{c_n}}(\mathbb{D}^n)$.

In particular
$$
\phi \in \mathcal{D}_{\frac{\alpha_1}{n}, \ldots, \frac{\alpha_n}{n}}(\mathbb{D}^n)
$$
\end{thm}
\begin{proof}
By Lemma \ref{embe}
$$
\| \phi \|_{\alpha_1/c_1, \ldots, \alpha_n/c_n}^2 \leq \prod_{i=1}^n \| \phi \|_{\alpha_i; i}^{1/c_i},
$$
and by Lemma \ref{hp} each factor on the right hand side is finite, which gives us the desired inclusion.
\end{proof}

\begin{corollary} \label{allRIFs}
Every RIF in $\mathbb{D}^n$ belongs to $\mathcal{D}_{1/n}(\mathbb{D}^n)$
\end{corollary}
\begin{proof}
By Lemma \ref{h1} every partial derivative of a RIF lies in $H^1(\mathbb{D}^n)$, so applying Theorem \ref{hp_embe} with $\alpha_k = 1$ for all $k$ finishes the proof.   
\end{proof}

\section{Examples and additional inclusions}

\subsection{$H^p$ and $\mathcal{D}_{\overline{\alpha}}$ comparisons}

\bigskip

Not surprisingly, Corollary \ref{allRIFs} is not sharp in general; there are RIF:s that belong to Dirichlet type spaces $\mathcal{D}_\alpha(\mathbb{D}^n)$ with $\alpha> 1/n$. In this section we consider several examples of RIF:s and use our previous results to determine when they lie in certain Dirichlet type spaces. Furthermore, in Theorem \ref{LMS11.3}, we provide a convenient method of determining slightly better regularity than that obtained from Corollary \ref{allRIFs} by providing a lower bound on the regularity of the RIF corresponding to $p$ if we know that $1/p$ lies in a certain Dirichlet type spaces. In some cases, the Dirichlet type norm of $1/p$ can be calculated in a straight-forward way using its power series representation. 

\begin{example} \label{fav_exmpl}
\normalfont
Consider the RIF on $\mathbb{D}^n$
$$
\phi_n(z) := \frac{n \prod_{i=1}^n z_i - \sum_{j = 1}^n \prod_{k \neq j} z_k}{n - \sum_{l=1}^n z_l}.
$$
In Example $2.5$ of \cite{SRIF}, it was shown that all partial derivatives of $\phi$ lie in $H^{p}(\mathbb{D}^n)$ for all $p < (n+1)/2$, and so, by Theorem \ref{hp_embe}, $\phi_n \in \mathcal{D}_{\alpha}$ for every such $\alpha < \frac{n+1}{2n} = \frac{1}{2} + \frac{1}{2n}$.

Although this estimate is not necessarily sharp, we can show that the above RIFs do not lie in the unweighted Dirichlet space $\mathcal{D}_1$ by considering only the diagonal elements in the power series norm.

First of all

\begin{equation} \label{theRIF}
\phi_n(z_1, ... ,z_n) = \left(n \sum_{i=1}^n z_i - \sum_{i=1}^n \prod_{j \neq i} z_j \right) \sum_{m=0}^\infty \frac{\left(\sum_{k=1}^n z_k \right)^m}{n^m}.
\end{equation}
And we have that
\begin{align*}
&\sum_{m=0}^\infty \frac{\left(\sum_{k=1}^n z_k \right)^m}{n^m} \\
= &\sum_{j_1 = 0}^\infty \sum_{j_2 = 0}^{j_1} \cdots \sum_{j_n = 0}^{j_{n-1}} \frac{1}{n^{j_1 - j_n}} {j_1 \choose j_2} {j_2 \choose j_3} \cdots {j_{n-1} \choose j_n} z_1^{j_1 - j_2} \cdots z_{n-1}^{j_{n-1} - j_n} z_n^{j_n},
\end{align*}
where the product of the binomial coefficients can be rewritten as 
$$
\frac{j_1 !}{(j_1 - j_2)! (j_2 - j_3)! \cdots (j_{n-1} - j_n)! j_n!}.
$$
By making the change of variables
$$
k_n = j_n
$$
and
$$
k_{n-1} = j_{n-1} - j_n, \quad k_{n-2} = j_{n-2} - j_{n-1}, \quad \ldots \quad, \quad k_1 = j_1 - j_2,
$$
the above series becomes
\begin{equation} \label{1/p}
\sum_{k_1 = 0}^\infty \sum_{k_2 = 0}^{\infty} \cdots \sum_{k_n = 0}^{\infty} \frac{1}{n^{\sum_{i=1}^n k_1}} \frac{\left( \sum_{i=1}^n k_i \right) !}{\prod_{i=1}^n (k_i)!} z_1^{k_1} \cdots z_{n-1}^{k_{n-1}} z_n^{k_n}.
\end{equation}
By plugging this into \eqref{theRIF}, we can find the Fourier coefficients of $\phi(z_1, .. , z_n)$. In particular we can find the coefficients on the diagonal. If we denote by $a_{k_1, k_2, .. , k_n}$ the Fourier coeffients of $\phi$, then we see that
\begin{align*}
a_{k_1+1, k_2+1, .. , k_n+1} = \frac{n \left( \sum_{i=1}^n k_i \right)!}{\prod_{i=1}^n (k_i!) n^{\sum_{i=1}^n k_i}} - \sum_{j=1}^n \frac{\left(1 + \sum_{i=1}^n k_i \right)!}{\prod_{i=1}^n (k_i!) (1+k_j) n^{1 + \sum_{i=1}^n k_i}} \\
= \frac{\left( \sum_{i=1}^n k_i \right)!}{\prod_{i=1}^n (k_i!) n^{\sum_{i=1}^n k_i}} \left( n -  \frac{1 + \sum_{i=1}^n k_i}{n} \sum_{j=1}^n \frac{1}{k_j + 1}\right),
\end{align*}
and for $l = k_1 = k_2 = \cdots = k_n$ we get
\begin{align*}
&a_{l+1, l+1, ... ,l+1} \\
= &\frac{(nl)!}{(l!)^n n^{nl}} \left( n - \frac{1+ nl}{n} \frac{n}{l+1} \right) = \frac{(nl)!}{(l!)^n n^{nl}} \frac{n-1}{l-1}.
\end{align*}
By Stirling's approximation, $(nl)!/((l!)^n n^{nl})$ is assymptotically comparable to
$$
c \frac{(nl)^{nl} \sqrt{nl}}{(l^l)^n n^{nl} \sqrt{l}^n} = c_2(n) l^{(1-n)/2}. 
$$
By just comparing with the diagonal elements, we get that
\begin{multline*}
\Vert \phi \Vert_\alpha^2 \geq \sum_{m=0}^\infty (2+m)^{n \alpha} a_{m+1,...,m+1}^2 \\
\geq k \sum_{m=2}^\infty (2+m)^{n \alpha} m^{1-n} (m-1)^{-2} \geq k \sum_{m=2}^\infty (m-1)^{n \alpha - 1 - n},
\end{multline*}
and the last series diverges if $n \alpha - 1 - n \geq -1 \iff n(\alpha -1) \geq 0 \iff \alpha \geq 1$.
\end{example}

The following example illustrates how Theorem \ref{hp_embe} can be applied to shift weight between the parameters $\alpha_1, \alpha_2$ and $\alpha_3$. More specifically, we show how a lot of regularity in one variable can be used to substantially raise the regularity in the other parameters at the expense of that one.   

\begin{example}
\normalfont
Consider the RIF $\phi = \tilde{p}/p$ on $\mathbb{D}^3$ obtained from
$$
p(z) = (2-z_1-z_2) + \frac{z_3}{2}(2 z_1 z_2 - z_1 - z_2).
$$
In Example $3.2$ of \cite{SRIF} it was shown that $\partial_{z_3} \phi$ is bounded, and $\partial_{z_1} \phi, \partial_{z_2} \phi \in H^{\alpha}(\mathbb{D}^3)$ for $\alpha < 3/2$.
By choosing
$$
c_1 = 2 + \epsilon/2, \quad c_2 = 2+ \epsilon/2,
$$
and 
$$
c_3 = 1 - c_1^{-1} - c_2^{-1},$$
and noting that $\partial_{z_3} \phi \in H^{c_3 \alpha_3}(\mathbb{D}^3)$ for all $\alpha_3 < \infty$, Theorem \ref{hp_embe} gives us that $\phi \in D_{\frac{3}{4+\epsilon}, \frac{3}{4+\epsilon} , \alpha_3}(\mathbb{D}^3)$. By choosing sufficiently small $\epsilon$, we see that $\phi \in D_{\alpha_1, \alpha_2, \alpha_3}(\mathbb{D}^3)$ for every choice of $\alpha_1, \alpha_2 < 3/4$ and every $\alpha_3 < \infty$.

If we instead choose
$$
c_1 = 3/2, \quad c_2 = 3 + 3 \epsilon/ 2$$
and
$$
c_3 = 1 - c_1^{-1} - c_2^{-1},
$$
then Theorem \ref{hp_embe} shows that $\phi \in D_{1, \frac{1}{2+\epsilon}, \alpha_3}$ where $\alpha_3 < \infty$. By choosing $\epsilon$ sufficiently small we see that $\phi \in D_{1, \alpha_2, \alpha_3}$ for every $\alpha_2 < 1/2$ and $\alpha_3 < \infty$.

Note that even though one partial derivative is bounded, the mixed partial derivative $\partial_{z_1} \partial_{z_2} \partial_{z_3} \phi$ is unbounded. In fact, direct calculation shows that
$$
\partial_{z_1} \partial_{z_2} \partial_{z_3} \phi(t,t,t) = \frac{t^4 - 5 t^3 + 8 t^2 - 4t + 4}{8 (t-1)^2 (t-2)^4},
$$
which has a singularity of degree $2$ as $t \ra 1$. 

\end{example}

One type of RIF:s that is fairly easy to examine are lifts from RIF:s on $\mathbb{D}^2$. The following examples illustrate how one can use that lifts will inherit the $H^p$ integrability of their partial derivatives from the lifted RIF together with Theorem \ref{hp_embe} in order to obtain inclusions in different Dirichlet type spaces where one shifts weight between the parameters $\alpha_1, \alpha_2$ and $\alpha_3$ in ways that are not immediately obvious from the power series norm. 

\begin{example}
\normalfont
Let $\phi(z_1,z_2)$ be an RIF which lies in $\mathcal{D}_{\alpha, \alpha}$ if and only if $\alpha < t/2$ and whose partial derivatives lie in $H^\mathfrak{p}(\mathbb{D}^2)$ for all $\mathfrak{p} < t$. Note that, from \cite{LMS} we know that such a $t$ must be of the form $t= 1 + 1/(2n)$ for some natural number $n$.

Then $f(x,y,z) := \phi(xy,z)$ is a RIF on $\mathbb{D}^3$. Furthermore, if $\phi(z_1, z_2) = \sum_{k,l} a_{k,l,} z_1^k z_2^l$, then $f(x,y,z) = \sum_{k,l} a_{k,l} (xy)^k z^l$, so $\hat{f}(m,n,s) = a_{k,l}$ if $m=n=k$ and $s=l$, and $0$ otherwise. It follows that
$$
\Vert f \Vert_{\alpha_1, \alpha_2, \alpha_3}^2 = \sum_{k,l} (1+k)^{\alpha_1 + \alpha_2} (1+l)^{\alpha_3} |a_{k,l}|^2.
$$
If $\alpha_3 = \alpha_1 + \alpha_2$, then we know that this series converges if and only if $\alpha_3 < t/2$, but we have freedom to choose $\alpha_1 \neq \alpha_2$. However, given the knowledge we have, we cannot, in any obvious way, raise $\alpha_3$ above $t/2$ by lowering $\alpha_1 + \alpha_2$. This however, can be done by considering $H^t-$integrability of the partial derivatives. We have that $\partial_z f = \partial_2 \phi(xy,z)$, and so
\begin{multline*}
\Vert \partial_z f \Vert_{H^\alpha}^\alpha = \int_{\mathbb{T}^3} | \partial_2 \phi(xy,z) |^\alpha |dx||dy||dz| \\
= \int_\mathbb{T} \left( \int_{\mathbb{T}^2} |\partial_2 \phi(xy,z))|^\alpha |dy||dz| \right) |dx|
\int_{\mathbb{T}} \Vert \partial_2 \phi \Vert_{H^\alpha}^\alpha |dx| = 2 \pi  \Vert \partial_2 \phi \Vert_{H^\alpha}^\alpha,
\end{multline*}
which is finite if and only if $\alpha < t$.

Similarly $\partial_x f = y \partial_1 \phi$ and $\partial_y f = x \partial_1 \phi$ will both lie in $H^\alpha$ if and only if $\alpha < t$ (the factors $x$ and $y$ have absolute value $1$ everywhere in the domain of integration). For example, by applying Theorem \ref{hp_embe} with
$$
c_3 = 2 \text{ and } 1/c_1 + 1/c_2 = 1/2, \quad c_1,c_2 > 0,
$$
we get that for every $p_0 < t$ and $\alpha_3 = p_0/c_3 = p_0/2$, and $\alpha_1 = p_0/c_1, \alpha_2 = p_0/c_2$. Here $\alpha_1 + \alpha_2 = \alpha_3 = p_0/2$, and furthermore we see immediately that $f \in D_{p_0/3,p_0/3,p_0/3}$ for every $p_0 < t$, which was not obvious from the power series norm.

A slightly more difficult example would be $f(x,y,z) = \phi(xz,yz)$. If $\phi(z_1,z_2) = \sum_{k,l} a_{k,l} z_1^k z_2^l,$ then $f(x,y,z) = \sum_{k,l} a_{k,l} x^k y^l z^{k+l}$, and so $\hat{f}(k,l,m) = a_{k,l}$ if $m = k+l$ and $0$ otherwise. For $\alpha_3 \leq 1$ it follows that
\begin{multline*}
\Vert f \Vert_{\alpha_1, \alpha_2, \alpha_3}^2 = \sum_{k,l} (1+k)^{\alpha_1} (1+l)^{\alpha_2} (1+k+l)^{\alpha_3}|a_{k,l}|^2 \\
\leq \sum_{k,l} (1+k)^{\alpha_1 + \alpha_3} (1+l)^{\alpha_2 + \alpha_3} |a_{k,l}|^2,
\end{multline*}
and the last series converges if $\alpha_1 + \alpha_3 < t/2$ and $\alpha_2 + \alpha_3 < t/2$. For $\alpha_1 = \alpha_2 = \alpha_3,$ we see that every possible value must satisfy $\alpha_i < t/4$. 
Note however that this approximation is not sharp. For example, along diagonal elements the original series has the weight $(1+k)^{\alpha_1 + \alpha_2} (1+2k)^{\alpha_3}$, whereas the last series has the weight $(1+k)^{\alpha_1 + \alpha_2 + 2 \alpha_3}$. In fact, we can obtain better estimates by applying Theorem \ref{hp_embe}. Similarly to the previous example, $\partial_x f = z \partial_1 \phi$ and $\partial_y f = z \partial_2 \phi$ will lie in $H^\alpha$ for $\alpha < t$. And for $\partial_z f= x \partial_1 \phi + y \partial_2 \phi$ we have
$$
\Vert \partial_z f \Vert_{H^\alpha} = \Vert x \partial_1 \phi + y \partial_2 \phi \Vert_{H^\alpha} \leq \Vert x \partial_1 \phi \Vert_{H^\alpha} + \Vert y \partial_2 \phi \Vert_{H^\alpha}, 
$$
which is finite for $\alpha < t$.
By applying Theorem \ref{hp_embe} we see that $f \in \mathcal{D}_{\alpha}(\mathbb{D}^3)$ for every $\alpha < t/3$.

We can apply this to the RIF:s from Example \ref{fav_exmpl}. The partial derivatives of the RIF
$$
\phi(z_1, z_2) = \frac{2 z_1 z_2 - z_1 - z_2}{2 - z_1 - z_2}
$$
lie in $H^t(\mathbb{D}^2)$ if and only if $t < 3/2$, and it was shown in Example $3$ of \cite{LMS} the $\phi \in \mathcal{D}_{\alpha}(\mathbb{D}^2)$ if and only if $\alpha < 3/4$. It follows that the RIF
$$
f_1(x,y,z) := \phi(xy,z) = \frac{2 x y z - xy - z}{2 - xy - z} 
$$
lies in $\mathcal{D}_{\alpha}$ for $\alpha < 1/2$ and in $\mathcal{D}_{\frac{\alpha}{2}, \frac{\alpha}{2}, \alpha},$ $\alpha < 3/4$. The last inclusion, also obtained from the $H^t$ estimates, is in fact sharp (as can be verified from the power series norm).
Furthermore
$$
f_2(x,y,z) := \phi(xz,yz) = \frac{2 x y z^2 - xz - yz}{2 - xz - yz}
$$
lies in $\mathcal{D}_{\alpha}(\mathbb{D}^3)$ for $\alpha < t/3 = 1/2$. 

\end{example}

\subsection{Comparisons for $\mathcal{D}_{\overline{\alpha}}$ norms}

The following lemma gives a more general way of comparing norms of different Dirichlet type spaces.

\begin{lemma}
Let $\alpha= (\alpha_1, \ldots, \alpha_n)$, $V = (v_1, \ldots, v_n)$ and $U = (u_1, \ldots, u_n)$ be multi-indices and let $p,q>1$ satisfy $p^{-1} + q^{-1} = 1$. Then
$$
\Vert f \Vert_{\alpha + V}^2 \leq \Vert f \Vert_{\alpha - U + pV}^{1/p} \Vert f \Vert_{\alpha + (q-1)U}^{1/q} 
$$ 
\end{lemma}
\begin{proof}
Let $f(z) = \sum_{k \in \mathbb{N}^n} a_k z^k$. Then by Hölder's inequality
\begin{multline*}
\Vert f \Vert_{\alpha' + V + U}^2 = \sum_k (1+k_1)^{\alpha'_1 + v_1 + u_1} \cdots (1+k_n)^{\alpha'_n + v_n + u_n} |a_k|^2 \\
= \sum_k ((1+k_1)^{\frac{\alpha'_1}{p} + v_1} \cdots (1+k_n)^{\frac{\alpha'_n}{p} + v_n} |a_k|^{\frac{2}{p}} ) ((1+k_1)^{\frac{\alpha'_1}{q} + u_1} \cdots (1+k_n)^{\frac{\alpha'_n}{q} + u_n} |a_k|^{\frac{2}{q}} ) \\
\leq \Vert f \Vert_{\alpha' + pV}^{1/p} \Vert f \Vert_{\alpha' + qU}^{1/q}.
\end{multline*}
Setting $\alpha = \alpha' - U$ finishes the proof.
\end{proof}

By applying this lemma inductively, that is, with $V = (q-1)U$, we obtain the following corollary.

\begin{corollary} \label{monsterfac}
Let $\alpha, V, U_1, U_2, \ldots U_{n-1}$ be multi-indices, and let $1< p_1, \ldots , p_{n-1}$, and $1< q_1, \ldots, q_n$ be the corresponding Hölder conjugates. Then
$$
\Vert f \Vert_{\alpha + V} \leq \Vert f \Vert_{\alpha + V'_1}^{1/(2p_1)} \Vert f \Vert_{\alpha + V'_2}^{1/(2^2 p_2 q_1)} \cdots \Vert f \Vert_{\alpha + V'_{n-1}}^{1/(2^{n-1} p_{n-1} q_{n-2} \cdots q_1)} \Vert f \Vert_{\alpha + V'}^{1/(2^{n-1} q_{n-1} \cdots q_1)},
$$
where
$$
V'_1 = -U_1 + p_1 V,\quad V'_{k+1} = -U_{k+1} + p_{k+1}(q_k - 1)V'_k
$$
and
$$
V' = (q_{n-1}-1)U_{n-1}.
$$
\end{corollary}

If we instead want to find $U_k$:s such that  $\Vert f \Vert_{\alpha + V}$ can be compared to the product of the norms $\Vert f \Vert_{\alpha + V'_i}$ for a specific choice of $V'_i$:s, then these $U_k:s$ must satisfy the following recursion relation:
$$
U_1 = -V'_1+ p_1 V, \quad U_{k+1} = -V'_{k+1} + \frac{p_{k+1}}{p_k-1} U_k,
$$
where we have used that $q_k - 1 = 1/(p_{k}-1)$.

The problem, of course, is that we have no control over the final factor $\Vert f \Vert_{\alpha V'},$ where $V' = \frac{1}{p_{n-1}-1} U_{n-1}$.
However, note that if we choose $p_{1} = n$, and $p_{k+1} = p_{k} - 1$, then every factor $\frac{p_{k+1}}{p_k - 1} = 1$, and $V' = U_{n-1}$. 

We denote by $\bar{1}_{-k}$ the multi-index whose $k$:th entry is $0$, and has all other entries equal to $1$. If we now choose $V = (-c, \ldots, -c)$, and $V'_{k} = -2 \bar{1}_{-k}$ for $k = 1, \ldots, n-1$, then the corresponding recursion formula for the $U_k$:s becomes
$$
U_1 = 2 \bar{1}_{-1} + (-cn, \ldots, -cn), \quad U_{k+1} = 2 \bar{1}_{-(k+1)} + U_k, \quad k = 1, \ldots, n-2.
$$ 
This recursion is fairly easy to solve, and it follows that
$$
U_{n-1} = (2(n-2), \ldots, 2(n-2), 2(n-1)) + (-cn, \ldots -cn).
$$
By setting $c = 2(n-1)/n$, we see that
$$
V' = U_{n-1} = (-2, \ldots, -2, 0) = -2 \bar{1}_{-n}.
$$
By applying this in Corollary \ref{monsterfac}, we see that
$$
\Vert f \Vert_{\alpha + (-c, \ldots, -c)} \leq \prod_{i=1}^n \Vert f \Vert_{\alpha - 2 \bar{1}_{-i}}^{1/c_i} ,
$$
where $c= -\frac{2(n-1)}{n}$ and for some constants $c_i$. In particular, we obtain the following Lemma.

\begin{lemma} \label{LMS11.2}
If $f \in \bigcap_{i=1}^n \mathcal{D}_{\alpha - 2 \bar{1}_{-i}}$ then $f \in \mathcal{D}_{\alpha - c}$ where $c = 2 - \frac{2}{n}$.
\end{lemma}

Our next goal is to show that if $1/p \in \mathcal{D}_{\alpha}$, then the mixed partial derivative of the RIF corresponding to $p$ will satisfy the conditions of Lemma \ref{LMS11.2}, which can be used to obtain information about the RIF. But first, we need the following analog of Lemma $11.1$ in \cite{LMS}.

\begin{lemma} \label{LMS11.1}

If $f = \frac{q}{p}$ where $p$ is a polynomial with no zeros in $\mathbb{D}^n$, then
$$
f \in \mathcal{D}_\alpha \text{ if and only if } \frac{\partial f}{ \partial z_i} \in \mathcal{D}_{\alpha - 2 e_i}, \quad \text{for all } i=1, \ldots, n.
$$
\end{lemma} 

\begin{thm} \label{LMS11.3}
Let $\phi = \tilde{p}/p$ be an RIF on $\mathbb{D}^n$ and suppose $1/p \in \mathcal{D}_\alpha$ for some $\alpha < 0$. Then $\phi \in \mathcal{D}_{\alpha + 2/n}$.
\end{thm}
\begin{proof}
We can write $\frac{\partial \phi}{\partial z_i} = \frac{q}{p^2} = \frac{q}{p} \frac{1}{p}$, where, by the quotient rule, $q \in \langle p, \tilde{p} \rangle$, the ideal generated by $p$ and $\tilde{p}$. Since $\phi$ is an RIF, it follows that $q/p \in H^\infty(\mathbb{D}^n)$, and so is a multiplier of $\mathcal{D}_\alpha$. By our assumption that $1/p \in \mathcal{D}_{\alpha}$, it follows that $\frac{\partial \phi}{\partial z_i} \in \mathcal{D}_\alpha$ for every $i = 1, \ldots, n$. By taking the partial derivatives with respect to the other $n-1$ variables, and by applying Lemma \ref{LMS11.1}, it follows that
$$
\frac{\partial^n \phi}{\partial z_i \cdots \partial z_n} \in \mathcal{D}_{\alpha - 2 \bar{1}_{-i}} \quad \text{for every } i = 1, \ldots n.
$$ 
By Lemma \ref{LMS11.2}, it follows that $\frac{\partial^n \phi}{\partial z_i \cdots \partial z_n} \in \mathcal{D}_{\alpha - c}$, where $c = \frac{2(n-1)}{n}$. By applying Lemma \ref{LMS11.1} again, we get that $\phi \in \mathcal{D}_{\alpha - c + 2}$, and since 
$$
-c + 2 = -\frac{2(n-1)}{n} + \frac{2n}{n} = \frac{2}{n},
$$
this finishes the proof. 
\end{proof}

\subsection{Estimates using the Łojasiewicz inequality}

Since Theorem \ref{LMS11.3} is applicable for $\alpha < 0$, and since the integral norm for Dirichlet type spaces with negative parameter $\alpha$ only uses the absolute value of the function $f$ (and no derivatives), this suggests that knowledge about how rapidly $p(z) \ra 0$ when $z \in \mathbb{D}^n$ approaches $Z(p) \cap \mathbb{T}^n$ can be used to determine whether $1/p(z)$ is contained in a $\mathcal{D}_\alpha$ for some $\alpha < 0$. 

One way of quantifying the decay of $p$ is by using the \emph{Łojasiewicz inequality}, which, as expressed in \cite{loj}, states that for a real analytic function $f$ there exist positive constants $C$ and $\alpha$ such that
\begin{equation} \label{loj}
C |f(x)| \geq \text{dist}(x, Z)^\alpha,
\end{equation}
where $Z$ is the zero set of $f(x)$ and $\text{dist}(x,Z) := \inf\{|x - z|: z \in Z \}$.

We have the following theorem, which essentially states that if $\alpha$ is small enough, then inclusion of $1/p(z)$ in certain Dirichlet type spaces follow, which by Theorem \ref{LMS11.3} gives us information about the corresponding RIF.

\begin{thm} \label{loj_thm}
Suppose $\alpha <0$ and that $p(z)$ is a polynomial with only one zero on the boundary of $\mathbb{D}^n$. If there exists a
\begin{equation} \label{loj_cond}
q < \frac{1-\alpha}{2} n
\end{equation}
such that 
$$
|p(z)| \geq C \text{dist}(z, (1, \ldots, 1))^q,
$$
for some $C>0$, then
$$
\tilde{p}/p \in \mathcal{D}_{\alpha+2/n}.
$$
\end{thm}

\begin{remark}
\normalfont
Note that the exponent $q$ is related to the Łojasiewicz exponent - the smallest value of $\alpha$ such that \eqref{loj} holds for some $C > 0$ - although it might in fact be smaller since we only care about a small part of the zero set and we have special restrictions on how we may approach the zero set.
\end{remark}

\begin{proof}

Assume without loss of generality that this zero is at the point $(1, \ldots, 1)$.  By Łojasiewicz inequality we know that there exists an exponent $q$ such that
$$
|p(z)| \geq C \text{dist}(z, (1, \ldots, 1))^q,
$$

Recall that for $\alpha \leq 0$, an equivalent norm for $\mathcal{D}_\alpha$ is given by
\begin{equation} \label{norm<0}
\int_{\mathbb{D}^n} |f(z)|^2 \prod_{i=1}^n (1 - |z_i|^2)^{-1-\alpha} dA(z).
\end{equation}
We are interested in calculating
$$
\int_{\mathbb{D}^n} \frac{1}{|p(z)|^2} \prod_{i=1}^n (1 - |z_i|^2)^{-1-\alpha} dA(z).
$$
Since $|p(z)| \geq C \text{dist}(z, (1, \ldots, 1))^q$, we have that this is bounded by 
\begin{equation} \label{areaintegral}
C \int_{\mathbb{D}^n} \frac{1}{\text{dist}(z, (1, \ldots, 1))^{2q}} \prod_{i=1}^n (1 - |z_i|^2)^{-1-\alpha} dA(z).
\end{equation}

Since the integrand is bounded in $\mathbb{D}$ outside of any neighbourhood of $(1, \ldots, 1)$, the above integral converges if and only if the corresponding integral over $\mathbb{D}^n \cap D(1,\epsilon)^n$ converges, where $D(1,\epsilon)^n$ is understood to be the polydisk with radius $\epsilon$ for some small $\epsilon>0$ and which is centred around the point $(1, \ldots, 1)$.

This region can be parametrized by setting $z_j = 1 - r_j \cos(v_j) - i r_j \sin(v_j)$ where $0 \leq r_j < \epsilon$ and where $v_j$ is such that $1 - r_j \cos(v_j) - i r_j \sin(v_j) \in \mathbb{D}$. The integration limits for $v_j$ are found by solving
$$
x^2 + y^2 = 1, \quad (1-x)^2 + y^2 = r_j^2 \Ra x_j = 1 - r_j^2/2,
$$
and so
\begin{equation} \label{max_cos}
\cos(v_j) = \frac{r_j^2/2}{r_j} = r_j/2
\end{equation}
gives the maximal and minimal values of $v_j$, which means that
$$
- \arccos(r_j/2) \leq v_j \leq \arccos(r_j/2).
$$

This gives us that the above integral converges if and only if

\begin{equation} \label{big_integral}
\int_0^\epsilon \int_{-\cos^{-1}(\frac{r_1}{2})}^{\cos^{-1}(\frac{r_1}{2})} \cdots \int_0^\epsilon \int_{-\cos^{-1}(\frac{r_n}{2})}^{\cos^{-1}(\frac{r_n}{2})} \frac{\prod_{j=1}^n r_j \prod_{j=1}^n (2 r_j \cos(v_j) - r_j^2)^{-1-\alpha}}{|\sum_{j=1}^n r_j^2|^{q}} d v_n d r_n \cdots d v_1 d r_1,
\end{equation}
is finite, where $\prod_{j=1}^n r_j$ is the Jacobian, $\text{dist}(z, (1, \ldots, 1))^q = \left|\sum_{j=1}^n r_j^2 \right|^{q}$ and
$$
\prod_{j=1}^n (1-|z_j|^2)^{-1-\alpha} = \prod_{j=1}^n (2 r_j \cos(v_j) - r_j^2)^{-1-\alpha}.
$$
By the AM-GM inequality, we have that 
$$
\left|\sum_{j=1}^n r_j^2 \right| \geq n \left|\prod_{j=1}^n r_j^{2/n} \right|,
$$
and so
$$
\eqref{big_integral} \leq \int_0^\epsilon \int_{-\cos^{-1}(\frac{r_1}{2})}^{\cos^{-1}(\frac{r_1}{2})} \cdots \int_0^\epsilon \int_{-\cos^{-1}(\frac{r_n}{2})}^{\cos^{-1}(\frac{r_n}{2})} \frac{\prod_{j=1}^n r_j^{- \alpha} (2 \cos(v_j) - r_j)^{-1-\alpha}}{\left|\prod_{j=1}^n r_j^{2q/n} \right|} d v_n d r_n \cdots d v_1 d r_1.
$$
The above integral is separable, and thus converges if and only if
\begin{equation} \label{fishy}
\int_0^\epsilon \int_{-\cos^{-1}(\frac{r}{2})}^{\cos^{-1}(\frac{r}{2})} \frac{r^{-\alpha}(2 \cos(v) - r)^{-1-\alpha}}{\left| r^{2q/n} \right|} dv dr < \infty.
\end{equation}
We begin by analysing the integral with respect to $v$.

By making the change of variables $2 \cos(v) - r = t$, we see that
\begin{multline*}
\int_{-\cos^{-1}(\frac{r}{2})}^{\cos^{-1}(\frac{r}{2})} (2 \cos(v) - r)^{-1-\alpha} dv = 2 \int_{0)}^{\cos^{-1}(\frac{r}{2})} (2 \cos(v) - r)^{-1-\alpha} dv \\
= \int_0^{2-r}\frac{t^{-1-\alpha}}{\sqrt{1- (t+r)^2/4}} dt,
\end{multline*}
and by dividing up the integral into the intervals $[0,1]$ and $(1, 2-r]$, we see that
$$
\int_0^{1}\frac{t^{-1-\alpha}}{\sqrt{1- (t+r)^2/4}} dt \leq C_1 \int_0^{1} t^{-1-\alpha}dt < \infty
$$
since $\alpha < 0$, and 
$$
\int_1^{2-r}\frac{t^{-1-\alpha}}{\sqrt{1- (t+r)^2/4}} dt \leq  \int_1^{2-r}\frac{C_2}{\sqrt{1- (t+r)^2/4}} dt,
$$
and by making the change of variables $w= t+r$, this equals
$$
\int_{1+r}^2 \frac{C_2}{\sqrt{(1-w/2)(1+w/2)}} dw  < C < \infty.
$$

And so \eqref{fishy} holds if 
$$\int_0^\epsilon \frac{r^{- \alpha}}{ r^{2q/n} } dr < \infty \iff - \alpha - 2q/n > -1 \iff n > 2q/(1-\alpha).
$$

This means that every polynomial for which some exponent 
\begin{equation*}
q < \frac{1-\alpha}{2} n
\end{equation*}
works will satisfy that $1/p(z) \in \mathcal{D}_\alpha$. And so, by Theorem \ref{LMS11.3} the corresponding RIF will lie in $\mathcal{D}_{\alpha+2/n}$. 

This finishes the proof. 

\end{proof}

As the next example shows, there are indeed polynomials whose singular set is a single point that can be shown to satisfy \eqref{loj_cond}.

\begin{example}
\normalfont
Consider the polynomials 
$$
p_n(z) = n - \sum_{i=1}^n z_i.
$$

To see that Theorem \ref{loj_thm} is applicable, we must analyze for which choices of $q$ there exist a constant $C$ such that 
$$
\left|n - \sum_{i=1}^n z_i \right| \geq C \text{dist}(z, (1, \ldots, 1))^q.
$$
The problem is of course that the left hand side should not vanish faster than the right hand side when all $z_i \ra 1$, and so we write $z_i = 1 - r_i e^{i v_i}$, where of course $r_i \leq 2$ and the angles must be such that $1 - r_i e^{i v_i}$ lies in the unit disc.
The above inequality now becomes
$$
\left| \sum_{i=1}^n r_i e^{i v_i} \right| \geq C \left( \sum_{i=1}^n r_i^2 \right)^{q/2}.
$$
We will now analyze for which choices of $q$ we have that 
$$
\frac{\left| \sum_{i=1}^n r_i e^{i v_i} \right|}{\left( \sum_{i=1}^n r_i^2 \right)^{q/2}} \geq C > 0,
$$
which is equivalent to showing that the above statement holds for $\sum_{i=1}^n r_i^2 < \epsilon^2$ for some small $\epsilon>0$.

Clearly $\left| \sum_{i=1}^n r_i e^{i v_i} \right| \geq \left| \sum_{i=1}^n r_i \cos(v_i) \right|$, and since all terms are positive, we have that for fixed $r \in B_\epsilon(0)$, this is minimized by choosing the angles as to minimize each term. This in turn is done by choosing the angles as large (or as small) as possible. As we saw in \eqref{max_cos}, for fixed $r_i$ this is achieved by $\cos(v_i) = r_i/2$, and so it follows that
$$
\left| \sum_{i=1}^n r_i \cos(v_i) \right| \geq \left| \sum_{i=1}^n r_i^2/2 \right| = \sum_{i=1}^n r_i^2/2, 
$$
and so
$$
\frac{\left| \sum_{i=1}^n r_i e^{i v_i} \right|}{\left( \sum_{i=1}^n r_i^2 \right)^{q/2}} \geq \frac{\sum_{i=1}^n r_i^2/2}{\left( \sum_{i=1}^n r_i^2 \right)^{q/2}} \geq C > 0
$$
whenever $q \geq 2$.

It follows that
$$
\left|n - \sum_{i=1}^n z_i \right| \geq C \text{dist}(z, (1, \ldots, 1))^q,
$$
holds with $q=2$ in every dimension. Since, $q = 2 < n(1-\alpha)/2 \iff \alpha < 1- \frac{4}{n}$ we see that \eqref{loj_cond} holds for all
$$
\alpha < \min \left(0, 1- \frac{4}{n} \right),
$$  
and so, by Theorem \ref{LMS11.3} for $n \geq 4$ the corresponding RIF lies in 
$\mathcal{D}_\alpha(\mathbb{D}^n)$ for all $\alpha < 2/n$, and for $n=2,3$ the corresponding RIF lies in $\mathcal{D}_\alpha(\mathbb{D}^n)$ for $\alpha < t$, where
$$
t = 1 - \frac{4}{n} + \frac{2}{n} = 1-\frac{2}{n},
$$
which equals $0$ for $n=2$ and $1/3$ for $n=3$.

\end{example}

\section*{Acknowledgements}

The author thanks his PhD advisor Alan Sola for several valuable discussions and ideas.

\end{document}